\documentclass[leqno,12pt]{amsart}
\usepackage{amsfonts}
\usepackage{amsmath,amssymb,amsthm}

\everymath{\displaystyle}

\newcommand{\E}{\mathbb E}
\newcommand{\R}{\mathbb R}
\newcommand{\tr}{\mathrm{tr}}
\newcommand{\ds}{\displaystyle}
\newcommand{\manifold}[1]{\mathcal{#1}}
\newcommand{\M}{\manifold{M}}
\newcommand{\D}{\manifold{D}}

\newtheorem{theorem}{Theorem}[section]

\newtheorem{prop}[theorem]{Proposition}
\theoremstyle{definition}
\newtheorem{defn}[theorem]{Definition}
\theoremstyle{remark}
\newtheorem{rem}[theorem]{Remark}

\numberwithin{equation}{section}

\begin{document}

\title[Timelike Surfaces with Parallel Normalized Mean Curvature Vector]{Timelike Surfaces with Parallel Normalized Mean Curvature Vector Field in the Minkowski 4-Space}

\author{Victoria Bencheva, Velichka Milousheva}

\address{Institute of Mathematics and Informatics, Bulgarian Academy of Sciences,
Acad. G. Bonchev Str. bl. 8, 1113, Sofia, Bulgaria}
\email{viktoriq.bencheva@gmail.com}
\email{vmil@math.bas.bg}


\begin{abstract}

In the present paper, we study timelike surfaces with parallel normalized mean curvature vector field in the four-dimensional Minkowski space.  We introduce  special isotropic  parameters on each such surface, which we call canonical parameters, and prove a fundamental existence and uniqueness theorem stating that each timelike surface with parallel normalized mean curvature vector field is determined up
to a rigid motion in the Minkowski space by three geometric functions satisfying a system of three partial differential equations. In this way we minimize the number of functions and the number of partial
differential equations determining the surface, thus solving the Lund-Regge problem for this class of surfaces.

\end{abstract}

\maketitle

\section{Introduction}

In the local theory of surfaces both in Euclidean and pseudo-Euclidean spaces one of the basic problems is to find a minimal number of invariant functions, satisfying some natural conditions, that determine the surface up to a motion. This problem is known as the Lund-Regge problem \cite{Lund-Reg}. It is solved for minimal (or maximal) surfaces of co-dimension two in the Euclidean 4-space $\R^4$, the Minkoswki space $\R^4_1$ and the pseudo-Euclidean space $\R^4_2$. The surfaces with zero mean curvature in these spaces admit locally geometrically determined special isothermal parameters, called \textit{canonical},  such that the two main invariants (the Gaussian curvature and the normal curvature) of the surface satisfy a system of two partial differential equations called a \textit{system of natural PDEs}. The number of the invariant functions determining the surfaces and the number of
the differential equations are reduced to two. Moreover, the geometry of the corresponding zero mean curvature surface (minimal or maximal)  is determined  by the solutions of this system of natural PDEs. 

Special geometric parameters on minimal surfaces in $\R^4$ were introduced by T. Itoh  in \cite{Itoh}, and further, these parameters were used to prove that a minimal surface in $\R^4$ is determined up to a motion by two invariant functions satisfying a system of two PDEs \cite{Trib-Guad}.  Based on the canonical parameters, the system of natural PDEs was solved explicitly in terms
of two holomorphic functions  \cite{G-K-1}. The same problem was solved for maximal spacelike surfaces and minimal timelike surfaces in the Minkowksi space $\R^4_1$. Special isothermal parameters on maximal spacelike surfaces in $\R^4_1$ were 
introduced in \cite{Al-Pal} and it was proved that the local geometry of these surfaces is determined by
two invariant functions satisfying two PDEs. On the base of these canonical parameters the system of natural  PDEs of maximal spacelike surfaces was solved explicitly in \cite{G-K-2}. Minimal timelike surfaces in $\R^4_1$ were studied by G. Ganchev and the second author in \cite{G-M-IJM} and it was proved 
that they admit locally canonical parameters and their geometry is determined by two invariant functions, satisfying the following  system  of natural PDEs:
\begin{equation*}\label{Nat_Eq_K_kappa_R41-tl}
\begin{array}{lll}
\sqrt[4]{K^2+\varkappa^2\phantom{\big|}}\; \Delta^h \ln \sqrt[4]{K^2+\varkappa^2\phantom{\big|}}  &=& 2K; \\[1.5ex]
\sqrt[4]{K^2+\varkappa^2\phantom{\big|}}\; \Delta^h \arctan \ds\frac{\varkappa}{K} &=& 2\varkappa;
\end{array}  \qquad\quad K^2+\varkappa^2\neq 0,
\end{equation*} 
\noindent where $K$ is the Gaussian curvature,  $\varkappa$ is the curvature of the normal connection (the normal curvature), 
and $\Delta^h$ is the  hyperbolic Laplace operator.

Similar results were obtained for minimal Lorentz surfaces in the pseudo-Euclidean space with neutral metric 
 $\R^4_2$ in \cite{A-M-1}, \cite{G-K-3}, and \cite{K-M-1}. 

\vskip 2mm
So, the following natural question arises: \emph{How to introduce canonical parameters and obtain natural equations
for other classes of surfaces in 4-dimensional spaces?} 

This problem can be solved for the surfaces with parallel normalized mean curvature vector field -- another important class of surfaces both in Riemannian and pseudo-Riemannian geometry, since being a natural extension of the surfaces with parallel mean curvature vector field, they play an important role in Differential geometry and  Physics.

Surfaces with parallel normalized mean curvature vector field in the Euclidean 4-space $\R^4$ and spacelike surfaces with parallel normalized mean curvature vector field  in the Minkowski 4-space $\R^4_1$ are studied by G. Ganchev and the second author in \cite{G-M-Fil}. These classes of surfaces are described in terms of the so-called canonical parameters by three functions satisfying a system of three partial differential equations. Each surface with  parallel normalized mean curvature vector field in $\R^4$ is determined up to a motion by three functions $\lambda(u,v)$, $\mu(u,v)$ and $\nu(u,v)$ satisfying the following  system of partial differential equations 
\begin{equation*} \label{E:Eq0-1} 
\begin{array}{l}
\vspace{2mm}
\nu_u = \lambda_v - \lambda (\ln|\mu|)_v;\\
\vspace{2mm}
\nu_v = \lambda_u - \lambda (\ln|\mu|)_u;\\
\nu^2 - (\lambda^2 + \mu^2) = \frac{1}{2}|\mu| \Delta \ln |\mu|,
\end{array} 
\end{equation*}
where $\Delta$ denotes the Laplace operator.

The class of spacelike surfaces with parallel normalized mean curvature vector field in the Minkowski space $\R^4_1$ is described by three functions 
 $\lambda(u,v)$, $\mu(u,v)$ and $\nu(u,v)$ satisfying the following  system of partial differential equations 
\begin{equation*}  \label{E:Eq0-2}
\begin{array}{l}
\vspace{2mm}
\nu_u = \lambda_v - \lambda (\ln|\mu|)_v;\\
\vspace{2mm}
\nu_v = \lambda_u - \lambda (\ln|\mu|)_u;\\
\vspace{2mm}
\varepsilon(\nu^2 - \lambda^2 + \mu^2) = \frac{1}{2}|\mu| \Delta \ln |\mu|,
\end{array} 
\end{equation*}
where $\varepsilon = 1$ corresponds to the case the mean curvature vector field is spacelike, and $\varepsilon = - 1$ corresponds to the case  the mean curvature vector field is timelike.

\vskip 2mm
In the present paper, we focus our attention on the class of timelike surfaces with parallel normalized mean curvature vector field in the Minkowski 4-space $\R^4_1$. On each such surface we introduce 
special isotropic parameters $(u,v)$, which we call \textit{canonical}, that allow us to prove  the fundamental existence and uniqueness theorem in terms of three geometrically determined
functions. With respect to these parameters, the metric function and all invariants of the surface are expressed by these geometric functions.
The timelike surfaces with parallel normalized mean curvature vector field in $\R^4_1$ can be divided into three subclasses: 
\begin{itemize}
\item 
  surfaces satisfying $K - H^2 >0$;
\item 
  surfaces satisfying $K - H^2 <0$;
	\item 
 surfaces satisfying $K - H^2 =0$,
\end{itemize}
where $K$ is the Gauss curvature and $H$ is the mean curvature vector field. 

The timelike surfaces with parallel normalized mean curvature vector field in $\R^4_1$ for which $K - H^2 > 0$ are determined up to a rigid motion in $\R^4_1$ by three functions 
  $\lambda(u,v)$, $\mu(u,v)$, and $\nu(u,v)$ satisfying the following  system of partial differential equations: 
\begin{equation}  \label{E:Eq0-3}
\begin{array}{l}
\vspace{2mm}
\nu_u + \lambda_v = \lambda (\ln|\mu|)_v;\\
\vspace{2mm}
\lambda_u  - \nu_v = \lambda (\ln|\mu|)_u;\\
\vspace{2mm}
|\mu| (\ln |\mu|)_{uv} = -\nu^2 - (\lambda^2 + \mu^2).
\end{array} 
\end{equation}

The surfaces from the second subclass (characterized by the inequality $K - H^2 < 0$) are determined up to a rigid motion in $\R^4_1$ by three functions
$\lambda(u,v)$, $\mu(u,v)$, and $\nu(u,v)$ satisfying the system of PDEs: 
\begin{equation}  \label{E:Eq0-4}
\begin{array}{l}
\vspace{2mm}
\nu_u + \lambda_v = \lambda (\ln|\mu|)_v;\\
\vspace{2mm}
\lambda_u  + \nu_v = \lambda (\ln|\mu|)_u;\\
\vspace{2mm}
|\mu| (\ln |\mu|)_{uv} = -\nu^2 + (\lambda^2 + \mu^2).
\end{array} 
\end{equation}

 The surfaces from the third subclass (characterized by $K - H^2 =0$) are determined up to a rigid motion by three functions $\lambda(u,v)$, $\mu(u,v)$, and $\nu(u)$ 
satisfying:
\begin{equation}  \label{E:Eq0-5}
\begin{array}{l}
\vspace{2mm}
\nu_u + \lambda_v = \lambda (\ln|\mu|)_v;\\
\vspace{2mm}
|\mu| (\ln |\mu|)_{uv} = -\nu^2.
\end{array} 
\end{equation}

The above systems \eqref{E:Eq0-3},  \eqref{E:Eq0-4}, and \eqref{E:Eq0-5} are the background systems of natural partial differential equations describing
the three subclasses of timelike surfaces with parallel normalized mean curvature vector field in $\R^4_1$. 
In this way we solve the Lund-Regge problem for this class of surfaces in  $\R^4_1$.

\section{Preliminaries}

Let $\mathbb R^4_1$  be the four-dimensional Minkowski space  endowed with the metric
$\langle ., . \rangle$ of signature $(3,1)$.  The standard flat metric is given in local coordinates by
$dx_1^2 + dx_2^2 + dx_3^2 -dx_4^2.$

Let $\M = (\D , z)$ be a surface in $\R^4_1$, where $\D\subset\R^2$ and $z : \D \to \R^4_1$ is an immersion, i.e. 
$\M$ is locally parametrized by $\M: z = z(u,v), \, \, (u,v) \in {\mathcal D}$. The surface 
$\M$ is said to be
\emph{spacelike} (resp. \emph{timelike}), if $\langle ., . \rangle$ induces  a Riemannian (resp. Lorentzian) 
metric $g$ on $\M$. 

We use the notations $\widetilde{\nabla}$ and $\nabla$ for the Levi Civita connections on $\mathbb R^4_1$ and $\M$, respectively.
Thus, if $x$ and $y$ are vector fields tangent to $\M$ and $\xi$ is a normal vector field, then we have the following formulas of Gauss and Weingarten:
$$\begin{array}{l}
\vspace{2mm}
\widetilde{\nabla}_xy = \nabla_xy + \sigma(x,y);\\
\vspace{2mm}
\widetilde{\nabla}_x \xi = - A_{\xi} x + D_x \xi,
\end{array}$$
which define the second fundamental tensor $\sigma$, the normal connection $D$
and the shape operator $A_{\xi}$ with respect to $\xi$. In general, $A_{\xi}$ is not diagonalizable.

The mean curvature vector  field $H$ of $\M$ is defined as
$$H = \ds{\frac{1}{2}\,  \tr\, \sigma}.$$ 
A surface $\M$ is called \textit{totally geodesic} if its
second fundamental form vanishes identically. The surface is called \textit{minimal} if its
mean curvature vector vanishes identically, i.e.  $H=0$.

A normal vector field $\xi$ on a surface $\M$ is called \emph{parallel in the normal bundle} (or simply \emph{parallel}) if $D{\xi}=0$  \cite{Chen}.
The surface $\M$ is said to have \emph{parallel mean curvature vector field} if its mean curvature vector $H$ is parallel, i.e.
$D H =0$.  In the early 1970s, the surfaces with parallel mean curvature vector field in Riemannian space forms were classified by B.-Y. Chen \cite{Chen1} and S. Yau  \cite{Yau}. In 2009, B.-Y. Chen  classified spacelike surfaces with parallel mean
curvature vector field  in pseudo-Euclidean spaces with arbitrary codimension and later, Lorentz surfaces with parallel mean curvature vector field in  arbitrary pseudo-Euclidean space $\R^m_s$ were studied in \cite{Chen-KJM} and \cite{Fu-Hou}. Some classical and recent results on submanifolds with parallel mean curvature vector in Riemannian manifolds
as well as in pseudo-Riemannian manifolds are presented in the survey \cite{Chen-survey}.

The class of surfaces with parallel mean curvature vector field is naturally extended to the class of surfaces with parallel
normalized mean curvature vector field as follows: a surface  is said to have \textit{parallel normalized mean curvature vector field} if 
 $H$ is non-zero and  there exists a unit vector field in the direction of $H$ 
which is parallel in the normal bundle \cite{Chen-MM}.  
It is proved that every analytic surface with parallel normalized mean curvature
vector  in the Euclidean $m$-space $\mathbb{R}^{m}$ must either lie in a
4-dimensional space $\mathbb{R}^{4}$ or in a hypersphere of $\mathbb{R}^{m}$
as a minimal surface \cite{Chen-MM}.

Complete classification of biconservative surfaces with parallel normalized mean curvature
vector field in $\R^4$ is given in \cite{Sen-Turg-JMAA} and biconservative $m$-dimensional submanifolds with parallel normalized mean curvature
vector field in $\R^{n+2}
$ are studied in \cite{Sen}. Recently, 3-dimensional biconservative and biharmonic submanifolds with parallel normalized mean curvature
vector field in the Euclidean 5-space $\R^5$ have been studied in \cite{Sen-Turg-TJM}.

\vskip 2mm
Let $\M: z=z(u,v), \,\, (u,v) \in \mathcal{D}$ $(\mathcal{D} \subset \mathbb R^2)$  be a local parametrization on a
timelike surface in $\mathbb R^4_1$.
The tangent space $T_p\M$ at an arbitrary point $p=z(u,v)$ of $\M$ is  spanned by $z_u$ and $z_v$. We use the standard denotations
$E(u,v)=\langle z_u,z_u \rangle, \; F(u,v)=\langle z_u,z_v
\rangle, \; G(u,v)=\langle z_v,z_v \rangle$ for the coefficients
of the first fundamental form and denote $W=\sqrt{|EG-F^2|}$. Without loss of generality we assume that $E<0$ and $G>0$.
Choosing an orthonormal frame field $\{n_1, n_2\}$ of the normal bundle, i.e. $\langle
n_1, n_1 \rangle =1$, $\langle n_2, n_2 \rangle = 1$, $\langle n_1, n_2 \rangle = 0$, we can write the following derivative formulas:
\begin{equation}\label{E:Eq-1}
\begin{array}{l}
\vspace{2mm} \widetilde{\nabla}_{z_u}z_u=z_{uu} = - \Gamma_{11}^1 \, z_u +
\Gamma_{11}^2 \, z_v + c_{11}^1\, n_1 + c_{11}^2\, n_2;\\
\vspace{2mm} \widetilde{\nabla}_{z_u}z_v=z_{uv} = - \Gamma_{12}^1 \, z_u +
\Gamma_{12}^2 \, z_v + c_{12}^1\, n_1 + c_{12}^2\, n_2;\\
\vspace{2mm} \widetilde{\nabla}_{z_v}z_v=z_{vv} = - \Gamma_{22}^1 \, z_u +
\Gamma_{22}^2 \, z_v + c_{22}^1\, n_1 + c_{22}^2\, n_2;\\
\end{array} 
\end{equation}
where $\Gamma_{ij}^k$ are the Christoffel's symbols and the functions $c_{ij}^k, \,\, i,j,k = 1,2$  are defined by
$$\begin{array}{lll}
\vspace{2mm}
c_{11}^1 = \langle z_{uu}, n_1 \rangle; & \qquad  c_{12}^1 = \langle z_{uv}, n_1 \rangle; &  \qquad  c_{22}^1 = \langle z_{vv}, n_1 \rangle; \\
\vspace{2mm}
c_{11}^2 = \langle z_{uu}, n_2 \rangle; & \qquad  c_{12}^2 = \langle z_{uv}, n_2 \rangle;& \qquad c_{22}^2 = \langle z_{vv}, n_2 \rangle.
\end{array}$$

It is obvious, that  the surface $\M$ lies in a two-dimensional plane if and only if
it is totally geodesic, i.e. $c_{ij}^k=0$ for all $i,j,k = 1, 2.$ 
So, further we assume that at least one of the coefficients $c_{ij}^k$ is not
zero.

Let us consider the following determinants:
$$
\Delta_1 = \left\vert%
\begin{array}{cc}
\vspace{2mm}
  c_{11}^1 & c_{12}^1 \\
  c_{11}^2 & c_{12}^2 \\
\end{array}%
\right\vert, \quad
\Delta_2 = \left\vert%
\begin{array}{cc}
\vspace{2mm}
  c_{11}^1 & c_{22}^1 \\
  c_{11}^2 & c_{22}^2 \\
\end{array}%
\right\vert, \quad
\Delta_3 = \left\vert%
\begin{array}{cc}
\vspace{2mm}
  c_{12}^1 & c_{22}^1 \\
  c_{12}^2 & c_{22}^2 \\
\end{array}%
\right\vert.
$$
At a given point $p \in M^2$, the \textit{first normal space} of $\M$  in $\R^4_1$, denoted by
$\rm{Im} \, \sigma_p$, is the subspace given by
$${\rm Im} \, \sigma_p = {\rm span} \{\sigma(x, y): x, y \in T_p \M \}.$$

It is obvious, that the condition $\Delta_1 = \Delta_2 = \Delta_3 = 0$  characterizes points at which
the first normal space  ${\rm Im} \, \sigma_p$ is one-dimensional. Such points are called \textit{flat} or \textit{inflection} points of the surface \cite{Lane, Little}.
E. Lane  proved in \cite{Lane}, that  every point of a surface in  a 4-dimensional affine space is an inflection point
if and only if the surface is developable or lies in a 3-dimensional space.
So, further we consider timelike surfaces in $\R^4_1$ that are free of inflection  points, i.e. we assume that $(\Delta_1, \Delta_2, \Delta_3) \neq (0,0,0)$.

\section{Canonical parameters on timelike surfaces with parallel normalized mean curvature vector field}

For a timelike surface $\M$ in $\R^4_1$, locally thete exists a coordinate system  $(u,v)$ such that  the metric tensor $g$ of $\M$ has the form \cite{Lar}:
\begin{equation*} \label{E:Eq-g}
g= - f^2(u, v)(du\otimes dv + dv\otimes du)
\end{equation*}
for some positive function $f(u, v)$. Let $z=z(u, v), (u, v) \in \mathcal{D} \, 
 (\mathcal{D} \subset \R^2)$ be such a local parametrization on $\M$. 
Then, the coefficients of the first fundamental form are
\begin{equation*}
E = \langle z_u, z_u \rangle = 0; \quad F = \langle z_u, z_v \rangle = - f^2(u, v); \quad G = \langle z_v, z_v \rangle = 0.
\end{equation*}
We consider the pseudo-orthonormal tangent frame field  given by $x=\ds{\frac{z_u}{f}}$, $y=\ds{\frac{z_v}{f}}$. Obviously,
 $\langle x, x \rangle = 0$,  $\langle x, y \rangle = -1$, $\langle y, y \rangle = 0$. Then, the mean curvature vector field $H$ of $\M$ is given by 
$$H = - \sigma(x, y).$$

In the case $H \neq 0$ (non-minimal surface), we can choose a unit normal vector field $n_1$ which is collinear with the mean curvature vector field $H$, i.e. $H = \nu n_1$  for a smooth function $\nu = || H ||$. Then, $\sigma (x,y) = - \nu n_1$.
We choose a unit normal vector field $n_2$ such that $\{n_1, n_2\}$ is an orthonormal frame field of the normal bundle ($n_2$ is determined up to orientation). Then we have the following formulas:
\begin{equation*}
\begin{array}{l} \label{E:Eq-3}
\vspace{2mm}
\sigma (x,x) = \lambda_1 n_1 + \mu_1 n_2; \\
\vspace{2mm}
\sigma (x,y) = -\nu n_1; \\
\vspace{2mm}
\sigma (y,y) = \lambda_2 n_1 + \mu_2 n_2,
\end{array}
\end{equation*} 
where $\nu \neq 0$, $\lambda_1, \mu_1, \lambda_2, \mu_2$ are  smooth functions determined by:
$$
\begin{array}{l}
\vspace{2mm}
\lambda_1 = \langle \widetilde{\nabla}_x x, n_1 \rangle; \qquad  \mu_1 = \langle \widetilde{\nabla}_x x, n_2 \rangle; \\
\vspace{2mm}
\lambda_2 = \langle \widetilde{\nabla}_y y, n_1 \rangle; \qquad  \mu_2 = \langle \widetilde{\nabla}_y y, n_2 \rangle.  
\end{array}
$$
Using that $\langle z_u, z_u \rangle = 0, \; \langle z_u, z_v \rangle = -f^2(u,v),  \; \langle z_v, z_v \rangle = 0$,
after differentiation we calculate the coefficients $\Gamma_{ij}^k$, \, $i,j,k = 1,2$:
\begin{equation}
\begin{array}{ll} \label{E:Eq-4}
\vspace{2mm}
\Gamma_{11}^1 = \frac{2f_u}{f}; & \qquad \Gamma_{11}^2 = 0;\\
\vspace{2mm}
\Gamma_{12}^1 = 0; & \qquad  \Gamma_{12}^2 = 0;\\
\vspace{2mm}
\Gamma_{22}^1 = 0; & \qquad \Gamma_{22}^2 = \frac{2f_v}{f}.
\end{array}
\end{equation}
Having in mind that $x=\ds{\frac{z_u}{f}}$, $y=\ds{\frac{z_v}{f}}$, from \eqref{E:Eq-1} and \eqref{E:Eq-4}, after calculations we obtain:
\begin{equation} 
\begin{array}{l} \label{E:Eq-5}
\vspace{2mm}
\nabla _x x = \frac{f_u}{f^2} \,x  \\
\vspace{2mm}
\nabla _x y = \quad\qquad -\frac{f_u}{f^2} \,y \\
\vspace{2mm}
\nabla _y x = -\frac{f_v}{f^2} \,x \\
\vspace{2mm}
\nabla _y y = \quad\qquad\,\,\, \frac{f_v}{f^2} \,y 
\end{array}
\end{equation}

We denote $\gamma_1 = \frac{f_u}{f^2} = x(\ln f)$, $\gamma_2 = \frac{f_v}{f^2} = y(\ln f)$. So, using equalities \eqref{E:Eq-1} and 
\eqref{E:Eq-5} we obtain the followimg derivative formulas:
\begin{equation}\label{E:DerivFormIsotr}
\begin{array}{l}
\vspace{2mm}
\widetilde{\nabla} _x x = \gamma_1 x \qquad\quad + \lambda_1 n_1 + \mu_1 n_2 \\
\vspace{2mm}
\widetilde{\nabla}_x y = \quad\quad -\gamma_1 y -\nu n_1 \\
\vspace{2mm}
\widetilde{\nabla}_y x = -\gamma_2 x \quad\quad -\nu n_1 \\
\vspace{2mm}
\widetilde{\nabla}_y y = \quad\quad\,\,\, \gamma_2 y \, \, +  \lambda_2 n_1 + \mu_2 n_2
\end{array}
\end{equation}

\begin{rem} 
The pseudo-orthonormal frame field $\{x,y,n_1,n_2\}$ is geometrically determined: $x,y$ are the two lightlike directions in the tangent space; $n_1$ is the unit normal vector field collinear with the mean curvature vector field $H$; $n_2$ is determined by the condition that   $\{n_1, n_2\}$ is an orthonormal frame field of the normal bundle ($n_2$ is determined up to a sign). We call this  pseudo-orthonormal frame field $\{x,y,n_1,n_2\}$ a \textit{geometric frame field} of the surface.
\end{rem}

Using \eqref{E:DerivFormIsotr}, we can easily derive the following derivative formulas for the normal frame field $\{n_1, n_2\}$:
\begin{equation}\label{E:DerivNormFormIsotr}
\begin{array}{l}
\vspace{2mm}
\widetilde{\nabla}_x n_1 = -\nu x  + \lambda_1 y \quad\quad + \beta_1 n_2 \\
\vspace{2mm}
\widetilde{\nabla}_y n_1 = \lambda_2 x -\nu y \quad\quad\,\,\,\,\, + \beta_2 n_2 \\
\vspace{2mm}
\widetilde{\nabla}_x n_2 = \quad\quad + \mu_1 y  -\beta_1 n_1 \\
\vspace{2mm}
\widetilde{\nabla}_y n_2 = \mu_2 x \quad\quad\,\,\,\, -  \beta_2 n_1
\end{array}
\end{equation}
where $\beta_1 = \langle \widetilde{\nabla}_x n_1 , n_2 \rangle$ and  $\beta_2 = \langle \widetilde{\nabla}_y n_1 , n_2 \rangle$. Formulas
\eqref{E:DerivFormIsotr} and \eqref{E:DerivNormFormIsotr} are the derivative formulas of the surface with respect to the pseudo-orthonormal frame field $\{x,y,n_1,n_2\}$ which is geometrically determined as explained above. 

The geometric meaning of the functions $\beta_1$ and  $\beta_2$ is revealed by the next two propositions.

\begin{prop} \label{P:parallel H}
Let $\M$ be   a timelike surface in the Minkowski space $\R^4_1$. Then, $\M$ has parallel mean curvature vector field if and only if $\beta_1 = \beta _2 =0$ and $\nu =const$.
\end{prop}

\begin{proof}
Let $\M$ be a timelike surface in $\R^4_1$  with geometric pseudo-orthonormal frame field $\{x,y,n_1,n_2\}$.  
It follows from \eqref{E:DerivNormFormIsotr} that for the normal mean curvature vector field $H = \nu n_1$ we have the formulas:
$$
\begin{array}{l} 
\vspace{2mm}
D_x H = x(\nu) n_1 + \nu \beta_1 n_2; \\ 
\vspace{2mm}
D_y H = y(\nu) n_1 + \nu \beta_2 n_2, 
\end{array}
$$
which imply that $H$ is parallel in the normal bundle if and only if $\beta_1 = \beta _2 =0$ and $\nu =const$.
\end{proof}

\begin{prop} \label{P:parallel n1}
Let $\M$ be   a timelike surface in the Minkowski space $\R^4_1$. Then, $\M$ has parallel normalized mean curvature vector field if and only if $\beta_1 = \beta _2 =0$ and $\nu \neq const$.
\end{prop}

\begin{proof}
Recall that $\M$  is  a surface with parallel normalized mean curvature vector field if 
 $H$ is non-zero (and non-parallel) and  there exists a unit vector field in the direction of $H$ 
which is parallel in the normal bundle. Since $n_1$ is collinear with $H$ and 
$$\begin{array}{l}
\vspace{2mm}
D_x n_1 = \beta_1 n_2; \\ 
\vspace{2mm}
D_y n_1 = \beta_2 n_2, 
\end{array}
$$
we conclude that $\M$  is  a surface with parallel normalized mean curvature vector field if and only if $\beta_1 = \beta _2 =0$ and $\nu \neq const$.
\end{proof}

\vskip 2mm
Further, we consider timelike surfaces with parallel normalized mean curvature vector field, i.e.  we assume that $\beta_1 = \beta _2 =0$ and $\nu \neq const$. For this class of surfaces we will introduce special, so-called canonical parameters, which we will prove to exist locally on each such surface. 

Using that $\beta_1 = \beta _2 =0$, from \eqref{E:DerivFormIsotr} and \eqref{E:DerivNormFormIsotr} we derive the following derivative formulas for the class of surfaces with parallel normalized mean curvature vector field:
\begin{equation}\label{E:DerivForm-MNMCVF}
\begin{array}{ll}
\vspace{2mm}
\widetilde{\nabla} _x x = \gamma_1 x \qquad\quad + \lambda_1 n_1 + \mu_1 n_2;  & \qquad \quad  \widetilde{\nabla}_x n_1 = -\nu x  + \lambda_1 y; \\
\vspace{2mm}
\widetilde{\nabla}_x y = \quad\quad -\gamma_1 y -\nu n_1; & \qquad \quad \widetilde{\nabla}_y n_1 = \lambda_2 x -\nu y;\\
\vspace{2mm}
\widetilde{\nabla}_y x = -\gamma_2 x \quad\quad -\nu n_1; & \qquad \quad \widetilde{\nabla}_x n_2 = \quad\quad + \mu_1 y;\\
\vspace{2mm}
\widetilde{\nabla}_y y = \quad\quad\,\,\, \gamma_2 y \, \, +  \lambda_2 n_1 + \mu_2 n_2; & \qquad \quad \widetilde{\nabla}_y n_2 = \mu_2 x.
\end{array}
\end{equation}

Further, we calculate the integrability conditions for this class of surfaces. Since the Levi Civita  connection $\widetilde{\nabla}$ of $\R^4_1$ is flat, we have 
\begin{equation} \label{E:R'}
\widetilde{R}(x,y,x) = 0; \quad \widetilde{R}(x,y,y) = 0; \quad \widetilde{R}(x,y,n_1) = 0; \quad \widetilde{R}(x,y,n_2) = 0,
\end{equation}
where 
$$\widetilde{R}(x,y,z) = \widetilde{\nabla}_x \widetilde{\nabla}_y z  - \widetilde{\nabla}_y \widetilde{\nabla}_x z - \widetilde{\nabla}_{[x,y]} z$$ 
 for arbitrary vector fields $x, y, z$. 
It follows from \eqref{E:DerivForm-MNMCVF} that the commutator $[x,y]$ is expressed as follows 
$$
[x,y] = \widetilde{\nabla}_x y - \widetilde{\nabla}_y x = \gamma_2 x - \gamma_1 y.
$$
Then, by use of formulas \eqref{E:DerivForm-MNMCVF}, we calculate:
\begin{equation} \label{E:R'-cal}
\begin{array}{ll}
\widetilde{R}(x,y,x) = & \left(-x(\gamma_2)-y(\gamma_1) - 2 \gamma_1 \gamma_2 + \nu^2 - \lambda_1 \lambda_2 - \mu_1 \mu_2\right) \,x - \\
\vspace{2mm}
& - \left(x(\nu)+y(\lambda_1) + 2 \gamma_2 \lambda_1 \right) \,n_1 - \left(y(\mu_1) + 2\gamma_2 \mu_1 \right) \,n_2; \\

\widetilde{R}(x,y,y) = & \left(x(\gamma_2)+y(\gamma_1) + 2 \gamma_1 \gamma_2 - \nu^2 + \lambda_1 \lambda_2 + \mu_1 \mu_2\right) \,y + \\
\vspace{2mm}
&  + \left(x(\lambda_2)+y(\nu) + 2 \gamma_1 \lambda_2 \right) \,n_1 + \left(x(\mu_2)+ 2\gamma_1 \mu_2 \right) \,n_2;\\

\widetilde{R}(x,y,n_1) = & \left(x(\lambda_2)+y(\nu) + 2 \gamma_1 \lambda_2 \right) \,x - \left(x(\nu)+y(\lambda_1) + 2 \gamma_2 \lambda_1 \right) \,y + \\
\vspace{2mm}
& + \left(\mu_1 \lambda_2 - \lambda_1 \mu_2 \right) \,n_2;\\

\widetilde{R}(x,y,n_2) = & \left(x(\mu_2) + 2\gamma_1 \mu_2 \right) \,x - \left(y(\mu_1) + 2\gamma_2 \mu_1 \right) \,y + \\
\vspace{2mm}
& + \left(\lambda_1 \mu_2 - \mu_1 \lambda_2 \right) \,n_1.

\end{array}
\end{equation}

Now, taking into consideration \eqref{E:R'} and \eqref{E:R'-cal}, we obtain the following integrability conditions:
\begin{equation}\label{E:integrCondIsotr1}
\begin{array}{l}
\vspace{2mm}
x(\lambda_2)+y(\nu) + 2 \gamma_1 \lambda_2=0; \\
\vspace{2mm}
x(\nu)+y(\lambda_1) + 2 \gamma_2 \lambda_1=0; \\
\vspace{2mm}
x(\mu_2) + 2\gamma_1 \mu_2=0; \\
\vspace{2mm}
y(\mu_1) + 2\gamma_2 \mu_1=0; \\
\vspace{2mm}
x(\gamma_2)+y(\gamma_1) + 2 \gamma_1 \gamma_2 - \nu^2 + \lambda_1 \lambda_2 + \mu_1 \mu_2=0 ;\\
\vspace{2mm}
\mu_1 \lambda_2-\lambda_1 \mu_2 =0. 
\end{array}
\end{equation}

\begin{rem} 
If we assume that both $\mu_1$ and $\mu_2$ are zero functions, i.e. $\mu_1(u,v)= 0$ and  $\mu_2(u,v)= 0$ for all $(u, v) \in \mathcal{D}$, then from \eqref{E:DerivForm-MNMCVF} we obtain that $\Delta_1 = \Delta_2 = \Delta_3 = 0$, which means that the surface consists of inflection points. Moreover, from $\widetilde{\nabla}_x n_2 = 0$ and $\widetilde{\nabla}_y n_2 = 0$, we get that the normal vector field $n_2$ is constant, which implies that the surface $\M$ lies in the three-dimensional Minkowski space $\R^3_1 = {\rm span} \{x,y,n_1\}$.
\end{rem} 

So, further we assume that $\mu_1^2 + \mu_2^2 \neq 0$ at least  in a sub-domain $\mathcal{D}_0$ of $\mathcal{D}$.
Without loss of generality we may assume that $\mu_1 \neq 0$. Then, from the last equality of \eqref{E:integrCondIsotr1} we obtain that
$\mu_1 \lambda_2=\lambda_1 \mu_2$, which implies $\lambda_2=\frac{\mu_2}{\mu_1} \lambda_1$. 

The Gauss curvature of the surface is defined by the following formula: 
$$K = \ds \frac{\langle R(x,y,y), x \rangle}{\langle x,x \rangle \langle y,y \rangle - \langle x,y \rangle^2}.$$
Now, using that $R(x,y,y) = \nabla_x \nabla_y y  - \nabla_y \nabla_x y - \nabla_{[x,y]} y$, from formulas \eqref{E:DerivForm-MNMCVF} we obtain 
$$R(x,y,y) = \left(x(\gamma_2)+y(\gamma_1) + 2 \gamma_1 \gamma_2\right)\, y,$$
and hence, the Gauss curvature $K$ is given by
$$K = x(\gamma_2)+y(\gamma_1) + 2 \gamma_1 \gamma_2.$$
Having in mind the fifth equality of \eqref{E:DerivForm-MNMCVF}, we obtain that the Gauss curvature of the surface $\M$ is expressed in terms of the functions $\nu$, $\lambda_1$, $\lambda_2$, $\mu_1$, $\mu_2$ as follows:
$$ K = \nu^2 - \lambda_1 \lambda_2 - \mu_1 \mu_2. $$

The last equality together with $\nu^2 = H^2$ (for simplicity we denote $H^2 = \langle H, H \rangle$) implies that 
$K - H^2 = -(\lambda_1 \lambda_2 + \mu_1 \mu_2).$
Using that  $\lambda_2=\frac{\mu_2}{\mu_1} \lambda_1$, we get 
$$K - H^2 = - \frac{\mu_2}{\mu_1}(\lambda_1^2 + \mu_1^2).$$
Hence, the surfaces with parallel normalized mean curvature vector field can be divided into two main classes:
\begin{itemize}
\item
$K - H^2 \neq 0$ (which is equivalent to $\mu_1 \mu_2 \neq 0$) in a sub-domain; 
\item
$K - H^2 = 0$ (which is equivalent to $\mu_1 \mu_2 = 0$) in a sub-domain.
\end{itemize}

\vskip 3mm
\subsection{Surfaces satisfying   $K - H^2 \neq 0$}.

\vskip 2mm
First we shall consider the case $K - H^2 \neq 0$, i.e. $\mu_1 \mu_2 \neq 0$. 
In this case, from the third and forth equalities of \eqref{E:integrCondIsotr1} we get:
$$\begin{array}{l}
\vspace{2mm}
 x(\ln |\mu_2|) = -2\gamma_1; \\ 
\vspace{2mm}
y(\ln |\mu_1|) = - 2\gamma_2.
\end{array} $$
On the other hand,  the functions $\gamma_1$ and $\gamma_2$ are expressed by the metric function $f$ as follows: 
$\gamma_1 = x(\ln f)$, $\gamma_2 =  y(\ln f)$. Hence, we obtain:
\begin{equation} \label{E:Eq101}
\begin{array}{l}
\vspace{2mm}
x(\ln f^2 |\mu_2|) = 0; \\
\vspace{2mm}
y(\ln f^2 |\mu_1|) = 0.
\end{array} 
\end{equation}
It follows from \eqref{E:Eq101} that 
the function  $f^2 |\mu_1|$ depends only on the parameter $u$, and the function $f^2 |\mu_2|$ depends only on $v$. Therefore, there exist smooth functions $\varphi (u) > 0$ and $\psi (v) > 0$ such that:
$$
f^2 |\mu_1| = \varphi (u); \quad f^2 |\mu_2| = \psi (v).
$$

We consider the following change of the parameters:
$$
\begin{array}{l}
\vspace{2mm}
\overline{u} = \int_{u_0}^u{\sqrt{\varphi (u)}}\, du + \overline{u}_0, \quad \overline{u}_0 = const\\
\vspace{2mm}
\overline{v} = \int_{v_0}^v{\sqrt{\psi (v)}}\, dv + \overline{v}_0, \quad \overline{v}_0 = const\\
\end{array} 
$$
Under this change of the parameters we obtain: 
$$
z_{\overline{u}} = \frac{z_u}{\sqrt{\varphi(u)}} = \frac{z_u}{f \sqrt{|\mu_1|}},
$$
$$
z_{\overline{v}} = \frac{z_v}{\sqrt{\psi(v)}} = \frac{z_v}{f \sqrt{|\mu_2|}},
$$
which imply that
$$
\langle z_{\overline{u}}, z_{\overline{u}} \rangle =0;  \quad
\langle z_{\overline{u}}, z_{\overline{v}} \rangle = - \frac{1}{\sqrt{|\mu_1||\mu_2|}}, \quad \langle z_{\overline{v}}, z_{\overline{v}} \rangle =0.
$$
Therefore, $(\overline{u}, \overline{v})$ are special isotropic parameters with respect to which the metric tensor of the surface is given by
$$
g = -\overline{f}^2 (\overline{u}, \overline{v}) (d\overline{u}\otimes d\overline{v} + d\overline{v}\otimes d\overline{u}),
$$
where the metric function $\overline{f}$ is expressed in terms of $\mu_1$ and $\mu_2$ as follows:
$$
\overline{f}(\overline{u}, \overline{v}) = \frac{1}{\sqrt[4]{|\mu_1||\mu_2|}}.
$$
With respect to the isotropic directions:
$$
\begin{array}{l}
\vspace{2mm}
\overline{x} = \frac{z_{\overline{u}}}{\overline{f}} =  \frac{\sqrt[4]{|\mu_1||\mu_2|}}{\sqrt{|\mu_1|}} x, \\
\vspace{2mm}
\overline{y} = \frac{z_{\overline{v}}}{\overline{f}} =  \frac{\sqrt[4]{|\mu_1||\mu_2|}}{\sqrt{|\mu_2|}} y, \\
\end{array}
$$
we have the flowing expressions for the second fundamental tensor $\sigma$:
\begin{equation*}
\begin{array}{l} \label{E:Eq102}
\vspace{2mm}
\sigma (\overline{x}, \overline{x}) = \frac{\sqrt{|\mu_1||\mu_2|}}{|\mu_1|} \,\sigma (x,x) = \lambda_1 \frac{\sqrt{|\mu_1||\mu_2|}}{|\mu_1|} \, n_1 + \mu_1 \frac{\sqrt{|\mu_1||\mu_2|}}{|\mu_1|} \, n_2; \\
\vspace{2mm}
\sigma (\overline{x}, \overline{y}) = \sigma (x,y) = - \nu  \,n_1;\\
\vspace{2mm}
\sigma (\overline{y}, \overline{y}) = \frac{\sqrt{|\mu_1||\mu_2|}}{|\mu_2|} \,\sigma (y,y) = \lambda_2 \frac{\sqrt{|\mu_1||\mu_2|}}{|\mu_2|} \, n_1 + \mu_2 \frac{\sqrt{|\mu_1||\mu_2|}}{|\mu_2|} \, n_2.
\end{array}
\end{equation*}

Since $\mu_1$ and $\mu_2$ are smooth functions and we consider a local theory, we may assume that $sign (\mu_1) = \varepsilon_1, \,\varepsilon_1 = \pm 1$ and $sign (\mu_2) = \varepsilon_2, \,\varepsilon_2 = \pm 1$ in some sub-domain. Now, using that 
$\frac{\lambda_2}{\lambda_1} = \frac{\mu_2}{\mu_1} = \frac{|\mu_2|}{|\mu_1|} \frac{\varepsilon_2}{\varepsilon_1}$, we get the formulas: 
$$
\begin{array}{l}
\vspace{2mm}
\sigma (\overline{x}, \overline{x}) = \lambda_1 \frac{\sqrt{|\mu_1||\mu_2|}}{|\mu_1|}\, n_1 + \varepsilon_1 \sqrt{|\mu_1||\mu_2|}\, n_2; \\
\vspace{2mm}
\sigma (\overline{y}, \overline{y}) = \lambda_1  \frac{\varepsilon_2}{\varepsilon_1} \frac{\sqrt{|\mu_1||\mu_2|}}{|\mu_1|}\, n_1 + \varepsilon_2 \sqrt{|\mu_1||\mu_2|}\, n_2. 
\end{array}
$$
Denoting $\overline{\lambda} = \lambda_1 \frac{\sqrt{|\mu_1||\mu_2|}}{|\mu_1|}$, $\overline{\mu} = \varepsilon_1 \sqrt{|\mu_1||\mu_2|}$, we obtain: 
\begin{equation*}
\begin{array}{l}
\vspace{2mm}
\sigma (\overline{x}, \overline{x}) = \overline{\lambda} \,n_1 + \overline{\mu} \,n_2; \\
\vspace{2mm}
\sigma (\overline{y}, \overline{y}) = \frac{\varepsilon_2}{\varepsilon_1} \overline{\lambda}\,n_1  + \frac{\varepsilon_2}{\varepsilon_1} \overline{\mu}\,n_2. 
\end{array}
\end{equation*}
Thus we conclude that there exist  two subcases:
\begin{enumerate}
\item
$\mu_1$ and $\mu_2$ have one and the same sign in the considered sub-domain, i.e. $\varepsilon_1 \varepsilon_2 = 1$, and hence we have $\sigma(\overline{x}, \overline{x}) = \sigma (\overline{y}, \overline{y})$.  \\

\item
 $\mu_1$ and $\mu_2$ have  opposite signs in the considered sub-domain, i.e. $\varepsilon_1 \varepsilon_2 = -1$, and hence we have $\sigma(\overline{x}, \overline{x}) = -\sigma (\overline{y}, \overline{y})$.\\
\end{enumerate}
Having in mind that $K - H^2 = - \frac{\mu_2}{\mu_1}(\lambda_1^2 + \mu_1^2)$, we get that the first subcase corresponds to $K - H^2 <0$, the second subcase corresponds to $K - H^2 >0$.

Hence, after the change of the parameters, we have the formulas:
\begin{equation*}\label{E:sigma2Isotr}
\begin{array}{l}
\vspace{2mm}
\sigma(\overline{x}, \overline{x}) = \overline{\lambda} \, n_1 + \overline{\mu} \, n_2 \\
\vspace{2mm}
\sigma(\overline{x}, \overline{y}) = -\nu \, n_1 \\
\vspace{2mm}
\sigma(\overline{y}, \overline{y}) = -\overline{\lambda} \, n_1 - \overline{\mu} \,n_2 
\end{array}, \quad \text{if} \; K - H^2 >0;
\end{equation*}
or 
\begin{equation*}\label{E:sigma1Isotr}
\begin{array}{l}
\vspace{2mm}
\sigma(\overline{x}, \overline{x}) = \overline{\lambda} \, n_1 + \overline{\mu} \, n_2 \\
\vspace{2mm}
\sigma(\overline{x}, \overline{y}) = -\nu \, n_1 \\
\vspace{2mm}
\sigma(\overline{y}, \overline{y}) = \overline{\lambda} \, n_1 + \overline{\mu} \,n_2 
\end{array}, \quad \text{if} \; K - H^2 <0.
\end{equation*}

In both cases ($K - H^2 > 0$ or $K - H^2 <0$), the metric function $\overline{f}$ is expressed by:
$$
\overline{f} = \frac{1}{\sqrt{|\overline{\mu}|}}.
$$

We introduce the notion of canonical isotropic parameters on a timelike surface with parallel normalized mean curvature vector field by the following definition.

\begin{defn} \label{D:canonicalIsotrParam}
 Let $\M$ be a timelike surface with parallel normalized mean curvature vector field in $\R^4_1$ and $K - H^2 \neq 0$. The isotropic parameters $(u,v)$ are said to be \textit{canonical} if the metric function $f$ is expressed by:
$$
f(u,v) = \frac{1}{\sqrt{|\mu|}}, \quad \mu \neq 0.
$$
\end{defn}

With the above considerations we have proved that:

\begin{prop}
Each timelike surface with parallel normalized mean curvature vector field satisfying $K - H^2 \neq 0$ locally admits canonical parameters.
\end{prop}

Let $\M: z = z(u,v), \,\, (u,v) \in {\mathcal D}$ be a timelike surface with parallel normalized mean curvature vector field satisfying $K - H^2 \neq 0$ and parametrized by isotropic canonical parameters $(u,v)$. 
With respect to canonical isotropic parametrization the derivative formulas of $\M$ take the form:
\begin{equation*}\label{E:DerivForm-MNMCVF-canonical}
\begin{array}{ll}
\vspace{2mm}
\widetilde{\nabla} _x x = \gamma_1 x \qquad\quad + \lambda n_1 + \mu n_2;  & \qquad \quad  \widetilde{\nabla}_x n_1 = -\nu x  + \lambda y; \\
\vspace{2mm}
\widetilde{\nabla}_x y = \quad\quad -\gamma_1 y -\nu n_1; & \qquad \quad \widetilde{\nabla}_y n_1 = -\varepsilon   \lambda x -\nu y;\\
\vspace{2mm}
\widetilde{\nabla}_y x = -\gamma_2 x \quad\quad -\nu n_1; & \qquad \quad \widetilde{\nabla}_x n_2 = \quad\quad + \mu y;\\
\vspace{2mm}
\widetilde{\nabla}_y y = \quad\quad\,\,\, \gamma_2 y \, \, -\varepsilon   \lambda n_1 -\varepsilon   \mu n_2; & \qquad \quad \widetilde{\nabla}_y n_2 = -\varepsilon   \mu x,
\end{array}
\end{equation*}
where $\varepsilon = 1$ in the case $K - H^2 > 0$, and $\varepsilon = - 1$ in the case $K - H^2 < 0$.

The geometric meaning of the canonical parametrization can be explained as follows: if $(u,v)$ are canonical isotropic parameters, then the canonical directions $x = \frac{z_u}{f}$ and $y= \frac{z_v}{f}$ satisfy  the relation:
\begin{equation*}
\begin{array}{l}
\vspace{2mm}
\sigma(x, x) = -\sigma (y, y), \quad \textrm {in the case} \;\; K - H^2 > 0;\\
\vspace{2mm}
\sigma(x, x) = \sigma (y, y), \quad \textrm {in the case} \;\; K - H^2 < 0.
\end{array}
\end{equation*}

Moreover, with respect to canonical isotropic parameters $(u,v)$, the functions $\gamma_1$ and $\gamma_2$ are expressed by:
\begin{equation} \label{E:Eq-104}
\gamma_1 = -\frac{|\mu|_u}{2\sqrt{|\mu|}}, \quad \gamma_2 = -\frac{|\mu|_v}{2\sqrt{|\mu|}}.
\end{equation}
 From integrability conditions \eqref{E:integrCondIsotr1}, in the case $\lambda_1 = \lambda$, $\lambda_2 = -\varepsilon \lambda$, $\mu_1 = \mu$, $\mu_2 = -\varepsilon \mu$, we get
\begin{equation*}\label{E:integrCondIsotr1-1}
\begin{array}{l}
\vspace{2mm}
x(\nu)+y(\lambda) + 2 \gamma_2 \lambda=0; \\
\vspace{2mm}
-\varepsilon x(\lambda)+y(\nu) -\varepsilon  2 \gamma_1 \lambda=0; \\
\vspace{2mm}
x(\gamma_2)+y(\gamma_1) + 2 \gamma_1 \gamma_2 - \nu^2 - \varepsilon (\lambda^2 + \mu^2)=0.
\end{array}
\end{equation*}
Then, having in mind \eqref{E:Eq-104}, from the equalities above we obtain
\begin{equation*} 
\begin{array}{l}
\vspace{2mm}
\nu_u + \lambda_v = \lambda (\ln|\mu|)_v;\\
\vspace{2mm}
\lambda_u - \varepsilon \nu_v = \lambda (\ln|\mu|)_u;\\
\vspace{2mm}
|\mu| \left(\ln |\mu|\right)_{uv} = - \nu^2 - \varepsilon (\lambda^2 + \mu^2).
\end{array} 
\end{equation*}

So, by introducing canonical parameters on a surface with  parallel normalized mean curvature vector field we manage to reduce up to three the number of functions and the number of partial differential equations. In the next section we shall prove that these three functions $\lambda$, $\mu$, and $\nu$ determine the surface up to a motion.

\vskip 3mm
\subsection{Surfaces satisfying   $K - H^2 = 0$}.

\vskip 2mm
Now we shall consider the case $K - H^2 = 0$, i.e. $\mu_1\mu_2 = 0$, $\mu_1^2 + \mu_2^2 \neq 0$. 
Without loss of generality we assume that $\mu_1 \neq 0$ and $\mu_2 = 0$  in a sub-domain $\mathcal{D}_0$. 
From $\mu_1\lambda_2 - \lambda_1 \mu_2 = 0$ it follows that $\lambda_2 = 0$, which implies that $K = \nu^2$.
In this case, the derivative formulas take the following form:
\begin{equation} \label{E:DerivForm-MNMCVF-canonical-2}
\begin{array}{ll}
\vspace{2mm}
\widetilde{\nabla} _x x = \gamma_1 x \qquad\quad + \lambda_1 n_1 + \mu_1 n_2;  & \qquad \quad  \widetilde{\nabla}_x n_1 = -\nu x  + \lambda_1 y; \\
\vspace{2mm}
\widetilde{\nabla}_x y = \quad\quad -\gamma_1 y -\nu n_1; & \qquad \quad \widetilde{\nabla}_y n_1 = \quad \quad  -\nu y;\\
\vspace{2mm}
\widetilde{\nabla}_y x = -\gamma_2 x \quad\quad -\nu n_1; & \qquad \quad \widetilde{\nabla}_x n_2 = \quad\quad  \mu_1 y;\\
\vspace{2mm}
\widetilde{\nabla}_y y = \quad\quad\,\,\, \gamma_2 y; & \qquad \quad \widetilde{\nabla}_y n_2 = 0.
\end{array}
\end{equation}
From integrability conditions \eqref{E:integrCondIsotr1} we get:
\begin{equation} \label{E:Eq-103}
\begin{array}{l}
\vspace{2mm}
y(\nu) = 0; \\
\vspace{2mm}
x(\nu) + y(\lambda_1) + 2\gamma_2 \lambda_1 = 0; \\
\vspace{2mm}
y(\mu_1) + 2\gamma_2 \mu_1 = 0; \\
\vspace{2mm}
x(\gamma_2) + y(\gamma_1) + 2 \gamma_1 \gamma_2 = \nu^2.
\end{array}
\end{equation}
So, the first equality of \eqref{E:Eq-103} implies that:
$$
\nu = \nu(u),
$$
and from the third one we get:
$$y(\ln|\mu_1|) = -2\gamma_2.$$
Having in mind that $\gamma_2 =  y(\ln f)$, we obtain
$$y(\ln (f^2 |\mu_1|)) = 0,$$
which implies that there exists a function $\varphi (u) > 0$ such that  $f^2 |\mu_1| = \varphi (u)$.

Consider the following change of the parameters:
$$
\begin{array}{l}
\vspace{2mm}
\overline{u} = \int_{u_0}^u \varphi (u) d u + \overline{u}_0, \quad \overline{u}_0 = const\\
\vspace{2mm}
\overline{v} = v + \overline{v}_0, \quad \overline{v}_0 = const
\end{array}
$$
Under this change of the parameters we obtain  
$$
\begin{array}{l}
\vspace{2mm}
z_{\overline{u}} = \frac{z_u}{f^2 |\mu_1|}; \\
\vspace{2mm}
z_{\overline{v}}  = z_v,
\end{array}
$$
which implies that
$$
\langle z_{\overline{u}}, z_{\overline{u}} \rangle = 0; \quad \langle z_{\overline{u}}, z_{\overline{v}} \rangle = -\frac{1}{|\mu_1|}; \quad 
\langle z_{\overline{v}}, z_{\overline{v}} \rangle = 0.
$$
Hence, $(\overline{u},\overline{v})$ are isotropic parameters with respect to which  the new metric function is:
$$
\overline{f} = \frac{1}{\sqrt{|\mu_1|}}.
$$
We consider the isotropic directions
$$
\begin{array}{l}
\vspace{2mm}
\overline{x} = \frac{z_{\overline{u}}}{\overline{f}} = \frac{x}{f\sqrt{|\mu_1|}}; \\
\vspace{2mm}
\overline{y} = \frac{z_{\overline{v}}}{\overline{f}} = f\sqrt{|\mu_1|} y. 
\end{array} 
$$
Then, the second fundamental tensor is expressed as follows:
$$
\begin{array}{l}
\vspace{2mm}
\sigma(\overline{x}, \overline{x}) = \frac{\lambda_1}{f^2|\mu_1|} n_1 + \frac{\mu_1}{f^2|\mu_1|}n_2; \\
\vspace{2mm}
\sigma(\overline{x}, \overline{y}) = -\nu n_1; \\
\vspace{2mm}
\sigma(\overline{y}, \overline{y}) =  0.
\end{array}
$$
Denoting 
$\overline{\lambda} = \frac{\lambda_1}{f^2|\mu_1|}$ and $\overline{\mu} = \frac{\mu_1}{f^2|\mu_1|}$,
we get
$$
\begin{array}{l}
\vspace{2mm}
\sigma(\overline{x}, \overline{x}) = \overline{\lambda} n_1 + \overline{\mu} n_2; \\
\vspace{2mm}
\sigma(\overline{x}, \overline{y}) = -\nu n_1; \\
\vspace{2mm}
\sigma(\overline{y}, \overline{y}) =  0.
\end{array}
$$
Note that $\overline{\mu} = \frac{\varepsilon}{f^2}$, where $\varepsilon = sign (\mu_1)$. Obviously, $\frac{\overline{\lambda}}{\overline{\mu}} = \frac{\lambda_1}{\mu_1}$.

So, in the case $K - H^2 = 0$, we can also introduce canonical isotropic parameters by the next definition.

\begin{defn} \label{D:canonicalParam}
 Let $\M$ be a timelike surface with parallel normalized mean curvature vector field in $\R^4_1$ and $K - H^2 = 0$. The isotropic parameters $(u,v)$ are said to be \textit{canonical} if the metric function $f$ is expressed by:
$$
f(u,v) = \frac{1}{\sqrt{|\mu|}}, \quad \mu \neq 0.
$$
\end{defn}

With the above considerations we have proved that:

\begin{prop}
Each timelike surface with parallel normalized mean curvature vector field satisfying $K - H^2 = 0$ locally admits canonical parameters.
\end{prop}

With respect to canonical isotropic parameters, in the case $K - H^2 = 0$, we  have derivative formulas  \eqref{E:DerivForm-MNMCVF-canonical-2}.
 From integrability conditions \eqref{E:integrCondIsotr1}, in the case $\lambda_1 = \lambda$, $\lambda_2 = 0$, $\mu_1 = \mu$, $\mu_2 = 0$, we get
$$\begin{array}{l}
\vspace{2mm}
x(\nu)+y(\lambda)+2\gamma_2\lambda = 0;\\ 
\vspace{2mm}
x(\gamma_2)+y(\gamma_1)+2\gamma_1 \gamma_2 = \nu^2,
\end{array}$$
which in view of \eqref{E:Eq-104} imply
$$\begin{array}{l}
\vspace{2mm}
\nu_u + \lambda_v = \lambda (\ln |\mu|)_v; \\
\vspace{2mm}
|\mu| (\ln |\mu|)_{uv} = -\nu^2.
\end{array}$$

Hence, in the case $K - H^2 = 0$, by introducing canonical parameters on a surface with  parallel normalized mean curvature vector field we manage to reduce the number of functions and the number of partial differential equations determining the surface.

\section {Fundamental Theorems}

Now we shall prove fundamental existence and uniqueness theorems for the class of timelike surfaces
with parallel normalized mean curvature vector field in terms of canonical parameters.

\begin{theorem}\label{T:Fundamental Theorem-1}
Let $\lambda(u,v)$, $\mu(u,v)$ and $\nu(u,v)$ be  smooth functions, $\mu  \neq 0$, $\nu \neq const$, 
 defined in a domain
${\mathcal D}, \,\, {\mathcal D} \subset {\R}^2$, and satisfying the conditions
\begin{equation} \label{E:Eq-fund-1}
\begin{array}{l}
\vspace{2mm}
\nu_u + \lambda_v = \lambda (\ln|\mu|)_v;\\
\vspace{2mm}
\lambda_u - \varepsilon \nu_v = \lambda (\ln|\mu|)_u;\\
\vspace{2mm}
|\mu| \left(\ln |\mu|\right)_{uv} = - \nu^2 - \varepsilon (\lambda^2 + \mu^2),
\end{array} 
\end{equation}
where $\varepsilon = \pm 1$. If $\{x_0, \, y_0, \, (n_1)_0,\, (n_2)_0\}$ is a pseudo-orthonormal frame at
a point $p_0 \in \R^4_1$, then there exists a subdomain ${\mathcal D}_0 \subset {\mathcal D}$
and a unique timelike surface
$\M: z = z(u,v), \,\, (u,v) \in {\mathcal D}_0$  with parallel normalized mean curvature vector field, such that $\M$ passes through $p_0$, $\{x_0, \, y_0, \, (n_1)_0,\, (n_2)_0\}$ is the geometric
frame of $\M$ at the point $p_0$,  the functions   $\lambda(u,v)$, $\mu(u,v)$, $\nu(u,v)$ are the geometric functions
of the surface, and $K - H^2 >0$ in the case $\varepsilon =  1$, resp. $K - H^2 <0$ in the case $\varepsilon = - 1$.
Furthermore, $(u,v)$ are canonical isotropic parameters of $\M$.
\end{theorem}

\begin{proof}
Let us denote $\gamma_1 =  -(\sqrt{|\mu|})_u$, $\gamma_2 =  -(\sqrt{|\mu|})_v$ and 
  consider the following
system of partial differential equations for the unknown vector
functions $x = x(u,v), \, y = y(u,v), \, n_1 = n_1(u,v), \,n_2 = n_2(u,v)$
in $\R^4_1$:
\begin{equation}
\begin{array}{ll} \label{E:Eq-system1}
\vspace{2mm} 
x_u = \frac{1}{\sqrt{|\mu|}} \left(\gamma_1\, x + \lambda\, n_1  + \mu\, n_2\right)
& \quad x_v = \frac{1}{\sqrt{|\mu|}} \left( -\gamma_2\, x - \nu\, n_1\right)\\
\vspace{2mm} 
y_u = \frac{1}{\sqrt{|\mu|}}\left(- \gamma_1\, y - \nu\, n_1 \right)
& \quad y_v = \frac{1}{\sqrt{|\mu|}} \left( \gamma_2\, y -\varepsilon \lambda \, n_1 -\varepsilon \mu \, n_2 \right)\\
\vspace{2mm} 
(n_1)_u = \frac{1}{\sqrt{|\mu|}}\left( - \nu\, x + \lambda\, y \right) &
\quad (n_1)_v =  \frac{1}{\sqrt{|\mu|}}\left( - \varepsilon \lambda\, x - \nu\, y \right) \\
\vspace{1mm} 
(n_2)_u = \frac{1}{\sqrt{|\mu|}} \left(  \mu\, y \right) &
\quad (n_2)_v =  \frac{1}{\sqrt{|\mu|}} \left( -  \varepsilon \mu\, x \right)
\end{array}
\end{equation}
We denote
$$\mathcal{F} =
\left(%
\begin{array}{c}
\vspace{1mm}
  x \\
  \vspace{1mm}
  y \\
  \vspace{1mm}
  n_1 \\
  \vspace{1mm}
  n_2 \\
\end{array}%
\right)\!\!; \;\;
\mathcal{A} = \frac{1}{\sqrt{|\mu|}}  \left(%
\begin{array}{cccc}
\vspace{1mm}
  \gamma_1 & 0 & \lambda  & \mu \\
  \vspace{1mm}
  0 & -\gamma_1 & -\nu & 0 \\
  \vspace{1mm}
  -\nu & \lambda & 0 & 0 \\
  \vspace{1mm}
  0 & \mu & 0 & 0 \\
\end{array}%
\right)\!\!; \;\;
\mathcal{B} = \frac{1}{\sqrt{|\mu|}}  
\left(%
\begin{array}{cccc}
\vspace{1mm}
  -\gamma_2 & 0 & -\nu & 0 \\
  \vspace{1mm}
  0 & \gamma_2 & -\varepsilon \lambda & -\varepsilon \mu \\
  \vspace{1mm}
  -\varepsilon \lambda & -\nu & 0 & 0 \\
  \vspace{1mm}
  -\varepsilon \mu & 0 & 0 & 0 \\
\end{array}%
\right)\!.$$
Then, system \eqref{E:Eq-system1} can be written in matrix form as follows:
\begin{equation}
\begin{array}{l} \label{E:Eq-105}
\vspace{2mm}
\mathcal{F}_u = \mathcal{A}\,\mathcal{F},\\
\vspace{2mm} 
\mathcal{F}_v = \mathcal{B}\,\mathcal{F}.
\end{array}
\end{equation}
The integrability conditions of system \eqref{E:Eq-105} are $\mathcal{F}_{uv} = \mathcal{F}_{vu}$,
i.e.
\begin{equation} \label{E:Eq-106}
\displaystyle{\frac{\partial a_i^k}{\partial v} - \frac{\partial b_i^k}{\partial u}
+ \sum_{j=1}^{4}(a_i^j\,b_j^k - b_i^j\,a_j^k) = 0, \quad i,k = 1,
\dots, 4,}
\end{equation}
 where by $a_i^j$ and $b_i^j$ we denote  the
elements of the matrices $\mathcal{A}$ and $\mathcal{B}$. Using  \eqref{E:Eq-fund-1}, one can check that
 equalities \eqref{E:Eq-106} are fulfilled. Hence, there exists a subdomain
$\mathcal{D}_1 \subset \mathcal{D}$ and unique vector functions $x
= x(u,v), \, y = y(u,v), \,n_1 = n_1(u,v)$, $n_2 = n_2(u,v), \,\, (u,v)
\in \mathcal{D}_1$, which satisfy system \eqref{E:Eq-system1} and the conditions
$$x(u_0,v_0) = x_0, \quad y(u_0,v_0) = y_0, \quad n_1(u_0,v_0) = (n_1)_0, \quad n_2(u_0,v_0) = (n_2)_0.$$

It can be proved that $x(u,v), \, y(u,v), \,n_1(u,v), \,n_2(u,v)$ form
a pseudo-orthonormal frame in $\R^4_1$ for each $(u,v) \in
\mathcal{D}_1$. Indeed, let us consider the following functions:
$$\begin{array}{lll}
\vspace{2mm}
  \varphi_1 = \langle x,x \rangle; & \qquad \varphi_5 =
  \langle x,y \rangle +1; & \qquad \varphi_8 = \langle y,n_1 \rangle; \\
\vspace{2mm}
  \varphi_2 = \langle y, y \rangle; & \qquad \varphi_6 =
  \langle x,n_1 \rangle; & \qquad \varphi_9 = \langle y,n_2 \rangle; \\
\vspace{2mm}
  \varphi_3 = \langle n_1, n_1 \rangle - 1; & \qquad \varphi_7 =
  \langle x,n_2 \rangle; & \qquad \varphi_{10} = \langle n_1,n_2 \rangle; \\
\vspace{1mm}
  \varphi_4 = \langle n_2,n_2 \rangle - 1; &   &  \\
\end{array}$$
defined for $(u,v) \in \mathcal{D}_1$. Having in mind  that $x(u,v), \,
y(u,v), \,n_1(u,v), \,n_2(u,v)$ satisfy \eqref{E:Eq-system1}, we obtain  the system
\begin{equation}
\begin{array}{lll} \label{E:Eq-107}
\vspace{2mm}
\displaystyle{\frac{\partial \varphi_i}{\partial u} = p_i^j \, \varphi_j},\\
\vspace{1mm} \displaystyle{\frac{\partial \varphi_i}{\partial v} =
q_i^j \, \varphi_j};
\end{array} \qquad i = 1, \dots, 10,
\end{equation}
where $p_i^j, q_i^j, \,\, i,j = 1, \dots, 10$ are
functions of $(u,v) \in \mathcal{D}_1$.  System \eqref{E:Eq-107} is a linear
system of partial differential equations for the functions
$\varphi_i(u,v)$, satisfying the conditions $\varphi_i(u_0,v_0) = 0$ for all $i = 1,
\dots, 10$, since $\{x_0, \, y_0, \, (n_1)_0,\, (n_2)_0\}$ is a pseudo-orthonormal frame. Therefore, $\varphi_i(u,v) = 0, \,\,i = 1, \dots, 10$ for
each $(u,v) \in \mathcal{D}_1$. Hence, the vector functions
$x(u,v), \, y(u,v), \,n_1(u,v), \,n_2(u,v)$ form a pseudo-orthonormal frame
in $\E^4_1$ for each $(u,v) \in \mathcal{D}_1$.

Finally, we consider the following system of partial differential equations for the vector function
$z(u,v)$:
\begin{equation} \label{E:Eq-108}
\begin{array}{lll}
\vspace{2mm}
z_u = \frac{1}{\sqrt{|\mu|}}\, x\\
\vspace{1mm} z_v = \frac{1}{\sqrt{|\mu|}}\, y
\end{array}
\end{equation}
It follows from equalities \eqref{E:Eq-fund-1} and \eqref{E:Eq-system1}  that the integrability
conditions $z_{uv} = z_{vu}$ of system \eqref{E:Eq-108}
 are fulfilled. Hence,  there exists a subdomain  $\mathcal{D}_0 \subset \mathcal{D}_1$ and
a unique vector function $z = z(u,v)$, defined for $(u,v) \in
\mathcal{D}_0$ and satisfying $z(u_0, v_0) = p_0$.

Now, we consider the surface $\M: z = z(u,v), \,\, (u,v) \in
\mathcal{D}_0$. Obviously, $\M$ is a timelike surface in $\R^4_1$. It follows from  \eqref{E:Eq-system1} that $\M$ has parallel normalized mean curvature vector field, since $H =\nu \,n_1$; $D_x n_1 = 0$ and $D_y n_1 =0$. Moreover, $(u,v)$ are canonical isotropic  parameters of $\M$, since $\langle z_u, z_v \rangle = - \frac{1}{|\mu|}$, and the metric function is $f = \frac{1}{\sqrt{|\mu|}}$. 

\end{proof}

\vskip 2mm

\begin{theorem}\label{T:Fundamental Theorem-2}
Let $\lambda(u,v)$, $\mu(u,v)$ and $\nu(u)$ be  smooth functions, $\mu  \neq 0$, $\nu \neq const$,
defined in a domain
${\mathcal D}, \,\, {\mathcal D} \subset {\R}^2$, and satisfying the conditions
\begin{equation} \label{E:Eq-fund-2}
\begin{array}{l}
\vspace{2mm}
\nu_u + \lambda_v = \lambda (\ln|\mu|)_v;\\
\vspace{2mm}
|\mu| \left(\ln |\mu|\right)_{uv} = - \nu^2.
\end{array} 
\end{equation}
If $\{x_0, \, y_0, \, (n_1)_0,\, (n_2)_0\}$ is a pseudo-orthonormal frame at
a point $p_0 \in \R^4_1$, then there exists a subdomain ${\mathcal D}_0 \subset {\mathcal D}$
and a unique timelike surface
$\M: z = z(u,v), \,\, (u,v) \in {\mathcal D}_0$  with parallel normalized mean curvature vector field, such that $\M$ passes through $p_0$, $\{x_0, \, y_0, \, (n_1)_0,\, (n_2)_0\}$ is the geometric
frame of $\M$ at the point $p_0$,  the functions   $\lambda(u,v)$, $\mu(u,v)$, $\nu(u)$ are the geometric functions
of the surface, and $K - H^2 =0$.  Furthermore, $(u,v)$ are canonical isotropic parameters of $\M$.
\end{theorem}

\begin{proof}
Let us   consider the following
system of partial differential equations for the unknown vector
functions $x = x(u,v), \, y = y(u,v), \, n_1 = n_1(u,v), \,n_2 = n_2(u,v)$
in $\E^4_1$:
\begin{equation}
\begin{array}{ll} \label{E:Eq-system2}
\vspace{2mm} 
x_u = \frac{1}{\sqrt{|\mu|}} \left(\gamma_1\, x + \lambda\, n_1  + \mu\, n_2\right)
& \qquad x_v = \frac{1}{\sqrt{|\mu|}} \left( -\gamma_2\, x - \nu\, n_1\right)\\
\vspace{2mm} 
y_u = \frac{1}{\sqrt{|\mu|}}\left(- \gamma_1\, y - \nu\, n_1 \right)
& \qquad y_v = \frac{1}{\sqrt{|\mu|}} \left( \gamma_2\, y\right)\\
\vspace{2mm} 
(n_1)_u = \frac{1}{\sqrt{|\mu|}}\left( - \nu\, x + \lambda\, y \right) &
\qquad (n_1)_v =  \frac{1}{\sqrt{|\mu|}}\left(- \nu\, y \right) \\
\vspace{1mm} 
(n_2)_u = \frac{1}{\sqrt{|\mu|}} \left(  \mu\, y \right) &
\qquad (n_2)_v =  0
\end{array}
\end{equation}
where  $\gamma_1 =  -(\sqrt{|\mu|})_u$ and $\gamma_2 =  -(\sqrt{|\mu|})_v$. It follows from equalities \eqref{E:Eq-fund-2} that the integrability conditions of system \eqref{E:Eq-system2} are fulfilled.

\vskip 1mm
Further, the proof follows the steps in the proof of Theorem \ref{T:Fundamental Theorem-1} and therefore, we are not going to give the details.

\end{proof}

\begin{rem}
We can introduce also canonical non-isotropic parameters  which in the case  $K - H^2 >0$ have the same geometric meaning as the canonical parameters of spacelike surfaces with parallel normalized mean curvature vector field in $\R^4_1$ and $\R^4$.
\end{rem}

\vskip 6mm 
\textbf{Acknowledgments:}
The  authors are partially supported by the National Science Fund, Ministry of Education and Science of Bulgaria under contract KP-06-N52/3.

\end{document}